\documentclass[11pt,reqno,oneside]{amsart}
\usepackage[a4paper, left=2cm, right=2cm, top=3cm, bottom=3cm]{geometry}
\usepackage{amsmath}
\usepackage{amsthm}
\usepackage{amsfonts}
\usepackage{amssymb}
\usepackage{xcolor}
\usepackage{enumitem}
\usepackage[mathcal]{euscript}
\setlist[enumerate]{label={\rm(\roman*)}}

\usepackage{mathtools}
\usepackage{hyperref}
\usepackage[numbers,sort&compress]{natbib}
\hypersetup{
	breaklinks=true,
	colorlinks=true,
	linkcolor=teal,
	citecolor=violet}
\usepackage{graphicx}
\usepackage[utf8]{inputenc}
\usepackage[T1]{fontenc}

\theoremstyle{plain}
\newtheorem{theorem}{Theorem}[section]

\newtheorem{proposition}[theorem]{Proposition}
\newtheorem*{proposition*}{Proposition}

\theoremstyle{definition}
\newtheorem{remark}[theorem]{Remark}

\newtheorem{example}[theorem]{Example}

\numberwithin{theorem}{section}
\numberwithin{equation}{section}

\DeclarePairedDelimiter\abs{\lvert}{\rvert}
\DeclarePairedDelimiter\nrm{\|}{\|}
\DeclarePairedDelimiter\set{\lbrace}{\rbrace}

\renewcommand\d{\mathrm{d}}
\newcommand\dd{\,\d}
\newcommand{\hra}{\hookrightarrow}
\makeatletter
\newcommand{\hrastar}{%
  \mathrel{\mathop{\hra}\limits^{
    \vbox to 0ex{\kern-2\ex@
    \hbox{$\scriptstyle*$}\vss}}}}
\makeatother

\newcommand{\R}{\mathbb{R}}

\newcommand{\RR}{\mathcal R}

\newcommand{\MM}{\mathcal M}

\newcommand{\ip}{{\frac{1}{p}}}
\newcommand{\iq}{{\frac{1}{q}}}
\newcommand{\qp}{{\frac{q}{p}}}
\newcommand{\pq}{{\frac{p}{q}}}

\newcommand{\RM}{(\RR,\mu)}

\newcommand{\MRM}{\MM\RM}

\usepackage{xspace}
\makeatletter
\DeclareRobustCommand\onedot{\futurelet\@let@token\@onedot}
\def\@onedot{\ifx\@let@token.\else.\null\fi\xspace}
\def\eg{e.g\onedot} 
\def\ie{i.e\onedot} 
\def\cf{cf\onedot} 
 
\def\ri{r.i\onedot} 
\makeatother

\makeatletter
\def\paragraph{\bigskip\@startsection{paragraph}{4}%
  \z@\z@{-\fontdimen2\font}%
  {\normalfont\bfseries}}
\makeatother

\input{widebar}

\title[Almost compact embeddings between Orlicz and Lorentz spaces]{Almost compact embeddings\\ between Orlicz and Lorentz spaces}

\author{V\'\i t Musil\textsuperscript{1}}
\address{\textsuperscript{1}%
Masaryk University,
Faculty of Informatics,
Department of Computer Science,
Botanick\'a 68a,
602 00 Brno,
Czech Republic}

\author{Lubo\v s Pick\textsuperscript{2}}
\address{\textsuperscript{2}%
Charles University,
Faculty of Mathematics and Physics,
Department of Mathematical Analysis,
So\-ko\-lo\-vsk\'a~83,
186~75 Praha~8,
Czech Republic}

\author{Jakub Tak\'a\v c\textsuperscript{2,3}}
\address{\textsuperscript{3}%
University of Warwick,
Mathematics Institute,
Zeeman Building,
CV4 7AL Coventry,
United Kingdom}

\email[Musil, corresponding author]{musil@fi.muni.cz}

\urladdr[Musil]{0000-0001-6083-227X}
\urladdr[Pick]{0000-0002-3584-1454}
\urladdr[Tak\'a\v c]{0000-0003-2158-7456}

\begin{document}

\begin{abstract}
    We characterize when an Orlicz space $L^A$ is almost compactly (uniformly absolutely continuously) embedded into a Lorentz space $L^{p,q}$ in terms of a balance condition involving parameters $p,q\in[1,\infty]$, and a Young function~$A$.
    In the course of the proof, we develop a new method based on an inequality of Young type involving the measure of level sets of a given function.
\end{abstract}

\date{}

\maketitle

\bibliographystyle{abbrvnat}

\section{Introduction and results statement}

We recall the notion of an \emph{almost-compact embedding}, sometimes also called uniformly absolutely continuous embedding, before we highlight its context in applications and formulate our result.

Fix, once and for all, a $\sigma$-finite measure space $(\RR,\mu)$.
Denote by $\MM = \MM(\RR,\mu)$ the set of all extended real valued $\mu$-measurable functions on $\RR$.
By saying that $(Y,\nrm{\cdot})$ is a (quasi)normed space of functions acting on $(\RR,\mu)$ we mean that $\nrm{\cdot}\colon\MM\to[0,\infty]$ is a functional satisfying the axioms of a (quasi)norm and $Y$ is the collection of all $f\in\MM$ such that $\nrm{f}<\infty$.
For a measurable set $E$, the characteristic function $\chi_E$ attains one on $E$ and zero otherwise.

\paragraph{Almost compactness}

A~set $\mathcal{S}\subset Y$ is \emph{almost-compact} if one has
\begin{equation*}
    \lim_{n\to \infty} \sup_{f\in \mathcal{S}} \nrm{f\chi_{E_n}} =0
\end{equation*}
for every sequence of sets $E_n\subset \RR$ with $\chi_{E_n}\to 0$ $\mu$-a.e.
An operator $T\colon X\to Y$ is \emph{almost-compact} if the image under $T$ of every bounded set in $X$ is almost compact in $Y$.
We will particularly focus on the cases when the operator $T$ is the identity operator.
Then we talk about \emph{almost-compact embeddings}.
Given another quasi-normed space $X$, we say that $X$ is \emph{almost-compactly embedded into $Y$}, written as $X\hrastar Y$, if the unit ball of $X$, denoted by $B_X$, is contained and almost-compact in $Y$.
Note that almost-compact sets are bounded, and so, in particular, $X\hrastar Y$ implies that $X\subset Y$ and that $B_X$ is bounded in $Y$, a fact that will henceforth be denoted by $X\hra Y$.

\paragraph{Significance of almost compactness}

The notion of almost-compact embedding serves as a general and rather universal tool used to obtain norm-compactness of operators and embeddings.
The general rule of thumb which works in all reasonable cases is that if a set in a function space is almost compact and compact with respect to convergence in measure, then it is norm-compact, \cf~\citet[Theorem~3.1]{Lenka}, see also~\citet{Fer:10}, or \citet[Chapter~1, Exercise~8]{BS}.
In relation to the compactness of certain kernel integral operators, this concept can already be traced to the work of \citet{LZ}, which takes advantage of almost-compactness.

Almost-compact sets in the space $L^1$ built upon a finite measure space are also sometimes called equi-integrable, and those are precisely the sets which are relatively weakly compact, see \eg \citet[Theorem~4.7.18]{Bog}.
Another characterisation of almost compact sets in $L^1$ can be obtained through de la Vall\'ee Poussin's theorem (\eg \citet[Chapter~1]{Rao:91}), which states that a set is almost-compact in $L^1$ if and only if it is contained in a strictly smaller Orlicz space.
This characterisation was one of the motivating factors in the early development of Orlicz spaces.

Recently, the main application of almost-compact embeddings seems to dwell in the study of compact Sobolev embeddings in various, very general, settings.
The principal reason for this is that Sobolev embeddings tend to be compact with respect to the convergence in measure, whence one step of the procedure comes for free.
For the classical Euclidean embeddings, we bring to attention the characterisations by \citet[Theorems~4.2 and~4.6]{Lenka2}, for analogous results in more general settings we refer to the work of \citet{Pus:06, Ker:08, Cur:07}.

Since Lorentz spaces arise as natural targets and domains in limiting and non-limiting Sobolev embeddings, characterisations of almost-compact embeddings involving these spaces are useful for proving compact Sobolev embeddings.
In particular, we are able to use present results to show that certain types of Sobolev embeddings can only ever be compact, \ie if a continuous embedding holds, it necessarily needs to be compact.

Different types of comparison of Lorentz and Orlicz spaces are of interest as both these types of spaces often arise in tasks in which Lebesgue spaces are insufficient, like limiting Sobolev embeddings, optimality results in critical cases etc., yet the two scales are ``incompatible'' in the sense that if a function space is both a Lorentz and an Orlicz space, then it is a Lebesgue space.
For example, \citet[Lemma~4.2]{Alb:17} prove a sufficient condition for the embedding $L^{A}\hra L^{p,q}$,
which they later use to investigate continuity properties and moduli of continuity of solutions to the $p$-Laplace equation.

\paragraph{The results}

Given a Young function $A$, we shall denote by $L^A$ the relevant Orlicz space on $(\RR,\mu)$.
For exponents $p,q\in (0,\infty]$, by $L^{p,q}$ we denote the two-parameter Lorentz space on $(\RR,\mu)$.
The precise definitions are recalled in Section~\ref{S:background}.
Our main aim is to characterize an~almost-compact embedding of an Orlicz space into a Lorentz space, namely \smash{$L^A \hrastar L^{p,q}$}.
\begin{theorem}\label{T:main}
    Suppose that $(\RR,\mu)$ is non-atomic and satisfies $\mu(\RR)<\infty$. Let $A$ be a Young function and assume $1\le q<p<\infty$. Then $L^A\hrastar L^{p,q}$ if and only if
        \begin{equation}\label{E:the-condition-fin}
			\int^\infty \biggl( \frac{t^p}{A(t)} \biggr)^{\frac{q}{p-q}} \frac{\d t}{t} < \infty.
		\end{equation}
\end{theorem}

The sufficiency of condition~\eqref{E:the-condition-fin} is the main and most innovative result of the paper with a~rather involved proof.
This result moreover allows us to prove a notably much more useful statement that the almost compact embedding \smash{$L^{A}\hrastar L^{p,q}$} is \emph{equivalent} to the continuous embedding $L^A \hra L^{p,q}$ when $1\le q<p$.
This fact, complemented with another characterising condition expressed in terms of a~specific modular inequality, is the content of the following theorem.

For the sake of completeness, we also discuss the case when $\mu(R)=\infty$, in which situation only continuous embeddings make sense.
We apply the convention $1/\infty=0$ here.

\begin{theorem}\label{T:embeddings-orlicz-to-lorentz}
    Suppose that $(\RR,\mu)$ is non-atomic and non-trivial. Let $A$ be a Young function and $p,q\in[1,\infty]$.
    \begin{enumerate}[label={\rm\Alph*.}, ref={\rm\Alph*}, parsep=1ex]
        \item\label{en:convex} Suppose $1\leq q < p < \infty$ and consider the following statements.
        \begin{enumerate}
	       \item \label{en:OL-emb}
	       The embedding $L^A\hra L^{p,q}$ holds.
	       \item \label{en:OL-ace}
	       The embedding $L^A\hrastar L^{p,q}$ holds.
	       \item \label{en:OL-int}
	       It holds that
		      \begin{equation}\label{E:the-condition}
			     \int^\infty_{\frac{1}{\mu(\RR)}} \biggl( \frac{t^p}{A(t)} \biggr)^{\frac{q}{p-q}} \frac{\d t}{t} < \infty.
		      \end{equation}
	       \item \label{en:OL-mod}
	       There exists a positive constant $C$ such that the inequality
		      \begin{equation*}
				    \|f\|_{p,q} \le C\left(\int_{\RR} A(|f|)\dd\mu\right)^{\frac{1}{p}}
		      \end{equation*}
		     holds for every $f\in L^A$.
        \end{enumerate}
        The statements \ref{en:OL-emb}, \ref{en:OL-int} and \ref{en:OL-mod} are equivalent. If $\mu(\RR)<\infty$, then all of the statements are equivalent.
        \item\label{en:non-convex} If $1\leq p\leq q \leq \infty$ and $\mu(\RR)<\infty$, then the following statements are equivalent.
            \begin{enumerate}
                \item \label{en2:OL-acLeb}
	           The embedding $L^A\hrastar L^{p}$ holds.
	          \item \label{en2:OL-acLor}
	           The embedding $L^A\hrastar L^{p,q}$ holds.
	          \item \label{en2:OL-lim} It holds that $p<\infty$ and $\lim_{t\to \infty} {A(t)}/{t^p}=\infty$.
            \end{enumerate}
    \end{enumerate}
\end{theorem}

\begin{remark}
Finiteness of the measure space for the characterization of $L^A \hrastar L^{p,q}$ is imperative.
In fact, if $\mu(\RR)=\infty$, then the embedding $L^{A} \hrastar L^{p,q}$ can never hold, which can be shown using a standard `traveling hill' argument.

The expression ${1}/{\mu(\RR)}$ in condition~\ref{en:convex}~\ref{en:OL-int} of Theorem~\ref{T:embeddings-orlicz-to-lorentz} is just a convenient way of distinguishing between the cases when $\mu(\RR)=\infty$ and $\mu(\RR)<\infty$.
If $\mu(\RR)<\infty$, then the value ${1}/{\mu(R)}$ can be replaced with any positive number, testing the convergence of the integral at infinity.
\end{remark}

We shall also provide a second set of results concerning almost-compact embeddings of Lorentz spaces into Orlicz spaces.
These statements are dual to those in Theorem \ref{T:embeddings-orlicz-to-lorentz}.
However, since some of the inequalities involve non-linear functionals, standard duality arguments cannot be used directly.

\begin{theorem}\label{T:embeddings-lorentz-to-orlicz}
    Suppose that $(\RR,\mu)$ is non-atomic and non-trivial. Let $A$ be a Young function and $r,s\in[1,\infty]$.
    \begin{enumerate}[label={\rm\Alph*.}, ref={\rm\Alph*}, parsep=1ex]
        \item\label{en:convex-dual} Suppose $1<r<s\le\infty$ and consider the following statements:
            \begin{enumerate}
                \item \label{en:OL-emb-dual}
	           The embedding $L^{r,s}\hra L^B$ holds.
	          \item \label{en:OL-ace-dual}
	           The embedding $L^{r,s}\hrastar L^B$ holds.
	          \item \label{en:OL-int-dual}
	           It holds that
		         \begin{equation}\label{E:the-condition-dual}
                     \int^{\infty}_{\frac{1}{\mu(\RR)}}\left(\frac{B(t)}{t^r}\right)^{\frac{s}{s-r}}\frac{dt}{t}<\infty,
		      \end{equation}
        in which we interpret the value $\frac{s}{s-r}$ as $1$ in the case when $s=\infty$.
	       \item \label{en:OL-mod-dual}
	          There exists a positive constant $C$      such that the inequality
		      \begin{equation*}
				    \left(\int_{\RR} B(|f|)\dd\mu\right)^{\frac{1}{r}}\le C\nrm{f}_{r,s}
		      \end{equation*}
		      holds for every $f\in  L^{r,s}(\RR,\mu)$.
            \end{enumerate}
        The statements \ref{en:OL-emb-dual}, \ref{en:OL-int-dual} and \ref{en:OL-mod-dual} are equivalent. If $\mu(\RR)<\infty$, then all of the statements are equivalent.
        \item \label{en-non-convex-dual}If $1\leq s\leq r < \infty$ and $\mu(\RR)<\infty$, then the following statements are equivalent.
            \begin{enumerate}
                \item \label{en2:OL-acLeb-dual}
	           The embedding $L^r\hrastar L^{B}$ holds.
	          \item \label{en2:OL-acLor-dual}
	           The embedding $L^{r,s}\hrastar L^{B}$ holds.
	          \item \label{en2:OL-lim-dual} It holds that $r>1$ and $\lim_{t\to \infty} {t^r}/{B(t)}=\infty$.
            \end{enumerate}
    \end{enumerate}
\end{theorem}

\begin{remark}
The question of almost compact embeddings between Orlicz and Lorentz spaces $L^{p,q}$ for $p\in(0,1)$ is trivial and not very interesting.
In fact, $L^A\hrastar L^{p,q}$ holds for every Orlicz space $L^A$.
Indeed, assuming $\mu(\RR)<\infty$, it holds that $L^1 \hrastar L^{p,q}$ for every $p\in (0,1)$ and any $q\in (0,\infty]$, hence the claim follows by the fact that $L^A\hra L^1$ for every Orlicz space $L^A$.
On the other hand, there is no embedding of a Lorentz space $L^{r,s}$ with $r\in(0,1)$ and arbitrary $s$ into an Orlicz space.
\end{remark}

\begin{remark}
The assumption of non-atomicity in Theorems~\ref{T:main}, \ref{T:embeddings-orlicz-to-lorentz} and \ref{T:embeddings-lorentz-to-orlicz} is superfluous for some of the implications.
Namely, in Theorem~\ref{T:main}, the integral condition \eqref{E:the-condition-fin} implies the almost-compact embedding \smash{$L^A\hrastar L^{p,q}$} on any finite measure space.
Furthermore, the integral condition \ref{en:OL-int} of Theorem \ref{T:embeddings-orlicz-to-lorentz}~\ref{en:convex} implies the continuous embedding \ref{en:OL-emb} on any $\sigma$-finite measure space and the almost-compact embedding \ref{en:OL-ace} on any finite measure space.
Analogous considerations hold also for Theorem~\ref{T:embeddings-lorentz-to-orlicz}.
\end{remark}

\paragraph{Applications to Compact Sobolev Embeddings}

As an interesting application of our main results, we shall now show that the compactness of certain Sobolev embeddings follows from their continuous companions.
For Sobolev spaces and embeddings we use the conventions and notation by \citet{Cia:15}, restricted to the case when the measure involved is just the Lebesgue $n$-dimensional measure.
Firstly, suppose $\Omega$ is a sufficiently regular (\eg John) \emph{bounded} domain in $\R^n$ for some $n\geq 2$. Considering the first-order Sobolev embedding on Lorentz spaces, one has, for any $p\in(1,n)$ and $q\in [1,\infty]$,
    \begin{equation}\label{E:Sob-emb-optimal}
        W^1 L^{p,q} (\Omega) \hra L^{p^*\!,\,q} (\Omega),
    \end{equation}
where $p^*={np}/({n-p})$.
In the sequel, we will omit writing out the domain $\Omega$ for brevity.
Moreover, both of the spaces $L^{p,q}$ and $L^{p^*\!,\,q}$ are optimal amongst rearrangement-invariant spaces in the sense that if $X$ is a rearrangement-invariant space such that $W^1 X \hra L^{p^*\!,\,q}$, then $X \hra L^{p, q}$, and when $Y$ is a rearrangement-invariant space satisfying $W^1 L^{p,q} \hra Y$, then $L^{p^*\!,\,q}\hra Y$.
The embedding, as well as the optimality of the target space, can be found in \cite[Theorem 6.9]{Cia:15}.
The optimality of the domain space seems not to be stated explicitly in the literature but follows from \cite[Theorem 5.11]{Edm:00} in combination with the classical Hardy inequality.

\begin{proposition} \label{P:Sobolev-domain}
Suppose $p\in (1,n)$ and $q\in [1,p)$. Then, for any Young function $A$, the following are equivalent:
\begin{enumerate}
    \item \label{Enum:Sob-emb1} The continuous embedding
        \begin{equation*}
            W^{1} L^A \hra L^{p^*\!,\,q}
        \end{equation*}
        holds;
    \item \label{Enum:Sob-emb2} The compact embedding
        \begin{equation*}
            W^{1} L^A \hra\hra L^{p^*\!,\,q}
        \end{equation*}
        holds;
    \item The condition \eqref{E:the-condition} holds with $\mu(\mathcal R)=|\Omega|$.
     \end{enumerate}
\end{proposition}

\begin{proof}
Recalling that Lorentz and Orlicz spaces are rearrangement invariant and, using optimality of function spaces appearing in  \eqref{E:Sob-emb-optimal}, we see that \ref{Enum:Sob-emb1} is equivalent to $L^{A}\hra L^{p,q}$.
Furthermore, from the characterisation of compact Sobolev embeddings via almost-compact embeddings~\citep[Theorems 4.2 and 4.6]{Lenka2},
we have that \ref{Enum:Sob-emb2} holds if an only if $L^A\hrastar L^{p,q}$.
The statement then follows by Theorem~\ref{T:embeddings-orlicz-to-lorentz}~\ref{en:convex}.
\end{proof}

On the target side, by an analogous argument via Theorem \ref{T:embeddings-lorentz-to-orlicz} \ref{en:convex-dual}, we obtain the following result.

\begin{proposition}\label{P:Sobolev-target}
Suppose $p\in (1,n)$ and $q\in (p^*,\infty]$.
Then for any Young function $B$, the following are equivalent:
\begin{enumerate}
    \item The continuous embedding
        \begin{equation*}
            W^{1} L^{p,q} \hra L^B
        \end{equation*}
        holds;
    \item The compact embedding
        \begin{equation*}
            W^{1} L^{p,q} \hra\hra L^B
        \end{equation*}
        holds;
    \item The condition
        \begin{equation*}
        \int^\infty \left(\frac{B(t)}{t^{p^*}}\right)^{\frac{q}{q-p^*}} \frac{\d t}{t} < \infty
        \end{equation*}
        holds, in which we interpret the value ${q}/({q-p^*})$ as $1$ in the case when $q=\infty$.
\end{enumerate}
\end{proposition}

\begin{example}\label{Ex:weak-Lebesgue}
As a particular special case of the previous proposition, one obtains a result for \emph{weak Lebesgue spaces} $L^{p,\infty}$. Indeed, given any $p\in (1, n)$ and any Young function $B$, one has
\begin{equation*}
    W^1 L^{p,\infty} \hra L^B
    \quad \text{if and only if }\quad
    W^1 L^{p,\infty} \hra\hra L^B.
\end{equation*}
\end{example}

\begin{remark}[Spaces near $L^1$]
In the limiting cases where the domain space is near $L^1$, that is $p=1$ in Propositions~\ref{P:Sobolev-domain} and~\ref{P:Sobolev-target} and Example~\ref{Ex:weak-Lebesgue}, several serious complications arise.
First, since $L^{1,q}$ for $q>1$ is not normed, one may not use any duality arguments.
Next, the characterisations of compact Sobolev embeddings via almost compact embeddings \cite[Theorems 4.2 and 4.6]{Lenka2} are not available.
Furthermore, even the ordinary Sobolev embeddings must be interpreted in a suitable way, \cf \citet{Tal:94}, which differs from the standard definition of an embedding of quasi-normed spaces.
\end{remark}

\begin{remark}[Spaces near $L^\infty$]
In the limiting cases where the target space is near $L^\infty$, that is $p=n$ in Propositions~\ref{P:Sobolev-domain} and~\ref{P:Sobolev-target} and Example~\ref{Ex:weak-Lebesgue}, we have the following.

If $q=1$, then the continuous embedding
\begin{equation*}
    W^1 L^{n,1} \hra L^\infty
\end{equation*}
holds~\cite{Ste:81}, but is not compact, see for instance~\cite[Theorem~4.2]{Cur:07} or~\cite[Theorem~1.2]{Ker:08}.
Moreover, both the domain and the target are optimal among rearrangement-invariant spaces, see \citep[Theorem~3.5]{CiP:98} for the embedding and optimality of the domain and \cite[Theorem 6.9]{Cia:15} for the optimality of the target.
By the same argument as in the proof of Proposition~\ref{P:Sobolev-domain}, we conclude that for every Young function $A$, the embeddings
\begin{equation*}
    W^1 L^A \hra L^\infty
    \quad\text{and}\quad
    W^1 L^A \hra\hra L^\infty
\end{equation*}
are equivalent and characterised by
\begin{equation*}
    \int^\infty \biggl( \frac{t^n}{A(t)} \biggr)^{\frac{1}{n-1}} \frac{\d t}{t} < \infty.
\end{equation*}
On the target side, every finite-valued Young function $B$ satisfies $L^B\not= L^\infty$ and hence $L^\infty\hrastar L^B$.
Therefore, using similar arguments as above, we conclude that
\begin{equation*}
    W^1 L^{n,1} \hra\hra L^B
    \quad\text{if and only if $B$ is finite-valued}.
\end{equation*}

If we consider the domain space to be $L^{n,q}$ with $q>1$, the situation is more complicated.
The optimal target space is \emph{not} a two-parameter Lorentz space, but the Lorentz-Zygmund space $L^{\infty,q,-1}$, \citep[Theorem 6.9]{Cia:15}.
Therefore our results on the target side cannot be used.
Unfortunately, even the domain space $L^{n,q}$ is not the optimal \ri space for the target space $L^{\infty,q,-1}$.
hence, neither of our results on the domain side is applicable.
The fact that $L^{n,q}$ is not the optimal \ri domain space for  $L^{\infty,q,-1}$ in a first-order Sobolev embedding is proved in the special case when $q=n$ in~\cite[Theorem~8.2]{Pic:98}, but the proof works almost verbatim for any $q>1$.
\end{remark}

In conclusion, when considering the limiting Sobolev embeddings our result does not provide the answers in many corner cases.
This opens an interesting question about the relationship between bounded and compact Sobolev embeddings involving Orlicz spaces in these situations.

\paragraph{Outline of the proofs}

We begin by establishing the sufficiency part of Theorem~\ref{T:main}, which has by far the most involved proof of all of the presented results.
The proof relies on a selection of new ideas and techniques, prime among them being the use of the recently discovered Young-type inequality~\citep[Proposition 2.1]{MPT}. The converse implication will follow from the proof of  Theorem~\ref{T:embeddings-orlicz-to-lorentz}~\ref{en:convex}.

To prove Theorem~\ref{T:embeddings-orlicz-to-lorentz}~\ref{en:convex}, we first show the equivalence of continuous embedding~\ref{en:OL-emb}
and modular condition~\ref{en:OL-mod}
and then the equivalence of integral condition~\ref{en:OL-int} and modular condition~\ref{en:OL-mod},
thus asserting the equivalence of \ref{en:OL-emb}, \ref{en:OL-int} and \ref{en:OL-mod}.
We show the result first in the case $\mu(\RR)=\infty$, and then describe the modifications that are necessary to obtain the proof when $\mu(\RR)<\infty$.
Trivially, almost compact embedding \ref{en:OL-ace} always implies its continuous companion~\ref{en:OL-emb}.
The fact that integral condition~\ref{en:OL-int} implies almost compact embedding~\ref{en:OL-ace} under the assumption that $\mu(\RR)<\infty$, is the content of the part of Theorem \ref{T:main} which has been already proved.
Note that one of the implications is not new, namely \ref{en:OL-int} $\Rightarrow$ \ref{en:OL-emb} can be found in the work of \citet[Lemma~4.2, Part~I]{Alb:17}.
In the case when $q=1$, the equivalence \ref{en:OL-emb} $\Leftrightarrow$ \ref{en:OL-ace} can be found in a more general form in our earlier work \cite[Theorem A.5]{MPT}.

The scheme of the proof for Theorem~\ref{T:embeddings-orlicz-to-lorentz} \ref{en:non-convex} is \ref{en2:OL-acLor} $\Rightarrow$ \ref{en2:OL-lim} $\Rightarrow$ \ref{en2:OL-acLeb} $\Rightarrow$ \ref{en2:OL-acLor}.
Some of these implications are trivial, and some are known.
Namely, the fact that limit condition~\ref{en2:OL-lim} is equivalent to the almost-compact embedding of an Orlicz space into a Lebesgue space \ref{en2:OL-acLeb} is a classical knowledge from the theory of Orlicz spaces, see e.g.~\citep[Chapter 5, Section 5.3, Proposition 7]{Rao:91}.
The fact that the almost-compact embedding \ref{en2:OL-acLor} implies the limit condition \ref{en2:OL-lim} is obtained via an elementary calculation, which we include for the sake of completeness.

The statements of Theorem \ref{T:embeddings-lorentz-to-orlicz} are, in a way, dual assertions to those we discussed above.
Some of the proofs are, however, still somewhat tricky as, to the best of our knowledge, there is no direct method of dualizing modular conditions as the expressions involved are \emph{not homogeneous} in~$f$.

\section{Prerequisite definitions and statements}
\label{S:background}

\paragraph{Measurable functions}
Given a function $f\in \MM$, we will firstly need its \emph{distribution function}, $f_*$, given by
\begin{equation*}
	f_*(\lambda)=\mu(\{x\in\RR:|f(x)|>\lambda\})
		\quad \text{for $\lambda\in[0,\infty]$},
\end{equation*}
and its \emph{non-increasing rearrangement}, $f^*$, given by
\begin{equation} \label{E:def-rearrangement}
	f^*(t) = \inf\set{\lambda\ge 0: f_*(\lambda) \le t}
		\quad\text{for $t\in[0,\infty]$}.
\end{equation}
Note that $f^*\colon [0,\infty] \to [0,\infty]$ is non-increasing and $|\{f^*> \lambda\}|=\mu(\{|f|>\lambda\})$ for all $\lambda\geq 0$. Here $|\cdot|$ denotes the Lebesgue measure on $\R$.

\paragraph{Generalized inverses}
Suppose $G\colon [0,\infty] \to [0,\infty]$ is non-decreasing.
We define the \emph{left-continuous inverse} of $G$ by
\begin{equation*}
    G^{-1}(t)=\inf\set{\tau\ge 0: G(\tau)\ge t}
	\quad\text{for $t\in[0,\infty]$.}
\end{equation*}
Observe that $G^{-1}$ is always left-continuous, non-decreasing on $(0,\infty]$ and $G^{-1}\bigl(G(t)\bigr) \le t$ for $t\in[0,\infty]$.
Similarly, if $F\colon[0,\infty] \to [0,\infty]$ is non-increasing, we define its left-continuous inverse by
\begin{equation*}
    F^{-1}(t)=\sup\{\tau\geq 0: \eta(\tau)\ge t\}
        \quad\text{for $t\in[0,\infty]$}.
\end{equation*}

\paragraph{Young functions}
We say that $A\colon [0,\infty] \to [0,\infty]$ is a Young function if it is convex, non-decreasing, left-continuous and satisfies $A(0+)=0$ and $A(\infty)=\infty$.
Note that these properties imply that $A$ is finite on a right neighbourhood of zero and $A(t)\to\infty$ as $t\to \infty$.
Given a Young function $A$, there exists a unique right-continuous function $a\colon [0,\infty) \to [0,\infty]$ satisfying
\begin{equation*}
    A(t)=\int_0^t a(s) \d s\quad \text{for all $t\in[0,\infty]$.}
\end{equation*}
We extend this function by defining $a(\infty)=\infty$ and call $a$ the \emph{right-continuous derivative} of $A$, a fact denoted by $A'=a$. From the fact that $a$ is non-decreasing, one immediately infers that
\begin{equation}\label{E:Young-basic-derivative-estimate}
    A(t)\leq t a(t) \leq A(2t) \quad \text{for all $t\in [0,\infty]$.}
\end{equation}
The \emph{Young conjugate} $\widetilde A$ of $A$ is given by
\begin{equation*}
    \widetilde A(t) = \sup\left\{st-A(s): s\in[0,\infty]\right\}\quad\text{for $t\in[0,\infty]$,}
\end{equation*}
and it satisfies
\begin{equation}\label{E:integral-representation-of-young-conjugate}
    \widetilde A(t)=\int_0^t a^{-1}(s) \d s\quad \text{for $t\in[0,\infty]$,}
\end{equation}
where $a^{-1}$ is the left-continuous inverse of $a$. Then $\widetilde A$ is also a~Young function whose Young conjugate is $A$.

\paragraph{Orlicz spaces}
Given $f\in \MM$ and a Young function $A$, the expression
\begin{equation*}
    \varrho_A(f)=\int_{\mathcal R} A(|f|) \dd \mu
\end{equation*}
is called the \emph{Orlicz modular} (of $f$ with respect to $A$).
We say that $f_n \to 0$ in the \emph{modular convergence} if the Orlicz modular of $f_n$ approaches zero as $n\to \infty$. The \emph{Luxemburg norm} is given by
\begin{equation*}
	\| f \|_{A}
		= \inf\biggl\{
				\lambda>0: \int_\RR A\Bigl(\frac{f}{\lambda}\Bigr)\dd \mu \leq 1
			\biggr\}
		\quad \text{for $f\in \MM$.}
\end{equation*}
The \emph{Orlicz space} is then the collection of functions $f\in \MM$ such that $\| f \|_{L^A}<\infty$ equipped with the Luxemburg norm.
It holds that
\begin{equation}
\label{E:Orlicz-balls}
	\|f\|_A \le 1
		\quad\text{if and only if}\quad
	\int_{\RR} A(|f|)\dd\mu \le 1.
\end{equation}
By Fubini's theorem, the Orlicz modular of $f\in\MM$ may be expressed via its distribution function~as
\begin{equation} \label{E:Orlicz-distribution-formula}
    \int_{\RR} A(\abs{f})\dd \mu = \int_{0}^{\infty}f_{*}(t)a(t)\dd t.
\end{equation}

Inequalities~\eqref{E:Young-basic-derivative-estimate} imply that if the Young function $A_0$ is defined by
\begin{equation*}
    A_0(t)=\int_{0}^{t}\frac{A(s)}{s}\dd s
    \quad\text{for $t\in[0,\infty)$,}
\end{equation*}
then $L^A=L^{A_0}$ with equivalent norms.
Since all of the statements we shall be proving are invariant under equivalent norms, we can always assume that the right continuous derivative of a given Young function is continuous.
Moreover, if $\mu(\RR)<\infty$ and $A$, $A_1$ are Young functions that are equal near infinity, then $L^A=L^{A_1}$ with equivalent norms.
Therefore, we can also assume any behaviour of the given Young function near zero.

\paragraph{Lorentz spaces}
Given two exponents $p,q\in(0,\infty]$, the \emph{Lorentz functional} is defined as
\begin{equation}
	\nrm{f}_{p,q}
		= \begin{cases}\displaystyle
				\left(
					\int_{0}^{\infty} \bigl[t^\ip f^{*}(t)\bigr]^q \frac{\dd t}{t}
				\right)^\iq
					& \text{if $q\in(0,\infty)$},
					\\[\bigskipamount]
				\sup\limits_{t\in(0,\infty)} t^\ip f^{*}(t)
					& \text{if $q=\infty$},
			\end{cases}
\end{equation}
for $f\in \MM$,
and the Lorentz space is given by $L^{p,q}=\set{f\in\MM: \nrm{f}_{{p,q}}<\infty}$. The Lorentz functional is equivalent to a norm provided $p=q=1$ or $p=q=\infty$ or $p\in(1,\infty)$ and $q\in[1,\infty]$. In these cases, $L^{p,q}$ is a Banach space. For any $p\in(0,\infty]$, one also has $L^{p,p}=L^p$. Note that $L^{p,q}\subset L^{p,r}$ provided $q\leq r$ irrespective of the underlying measure space \cite{BS}.

Lorentz norm of $f\in\MM$ can also be expressed via its distribution function as
\begin{equation} \label{E:Lorentz-distribution-formulas}
	\nrm{f}_{L^{p,q}}
		= \left(q\int_{0}^{\infty} f_*(t)^\qp t^{q-1}\dd t\right)^\iq.
\end{equation}

\paragraph{A special Young-type inequality}
We shall need a special type of the Young inequality discovered recently in~\cite[Proposition 2.1]{MPT}. We state it explicitly here, as it is the principal tool used in our~proof.
\looseness=-1
\begin{proposition*}
Let $A$ and $G$ be Young functions and let $v\colon (0,\infty)\to [0,\infty]$ be measurable. Denote by $g$ the right-continuous derivative of $G$. Then
\begin{equation} \label{E:OL-inequality}
	\int_{0}^\infty G^{-1}\bigl(f_*(t)\bigr)v(t)\dd t
		\le \int_{0}^\infty g^{-1} \left( \frac{v(t)}{\lambda a(t)} \right)v(t) \dd t
			+ \lambda \int_{0}^\infty A(\abs{f})
\end{equation}
for any $\lambda>0$ and any measurable $f$.
\end{proposition*}

\section{Proofs}

\begin{proof}[Proof of Theorem \ref{T:main}]
We shall prove that the condition~\eqref{E:the-condition-fin} implies the embedding $L^A\hrastar L^{p,q}$. The converse implication will be established within the course of the proof of Theorem~\ref{T:embeddings-orlicz-to-lorentz}~\ref{en:convex}.

Firstly, let us consider the case when $A$ attains the value infinity.
Then $L^A=L^\infty$ and $L^A\hrastar L^{p,q}$ obviously holds, since $p<\infty$ and, owing to~\citet[Theorem~5.2]{Lenka}, $L^\infty \hrastar X$ for any r.i.~space $X\neq L^{\infty}$.
Thus, we can suppose that $A$ attains only finite values, which in turn means that the right-continuous derivative $a$ of $A$ also attains only finite values on $[0,\infty)$.
Moreover, as mentioned above, we can (and will) assume that $a$ is continuous.
We will be done if we show
\begin{equation}\label{E:embedding-orlicz-to-lorentz-3}
		\lim_{r\to0_+} \sup \set*{\nrm{f}_{p,q}^{q} : \nrm{f}_{L^A}\le 1,\ f_*(0) \le r} =0.
\end{equation}
Indeed, the condition that $f_*(0)\leq r$ is equivalent to requiring that $f$ is supported on a set of measure not exceeding $r$, therefore \eqref{E:embedding-orlicz-to-lorentz-3} implies $L^A\hrastar L^{p,q}$.

Denote $G(t)=t^\pq$ for $t\ge 0$. Given $r>0$, define the function $G_r$ by
\begin{equation*}
    G_r(t) =
        \begin{cases}
            t^\pq &\text{if $t\in\bigl[0, r^\qp\bigr]$}
                \\
            \infty &\text{if $t\in\bigl(r^\qp,\infty\bigr]$.}
        \end{cases}
\end{equation*}
Then the inverse $G_r^{-1}$ of $G_r$ satisfies
\begin{equation*}
    G_r^{-1}(\tau) =
        \begin{cases}
            \tau^\qp &\text{if $\tau\in[0, r]$}
                \\
            r^\qp &\text{if $\tau\in(r,\infty]$,}
        \end{cases}
\end{equation*}
the left-continuous derivative $g_r$ of $G_r$ reads as
\begin{equation*}
    g_r(t) =
        \begin{cases}
            \pq t^{\pq-1}
                &\text{if $t\in\bigl[0, r^\qp\bigr]$}
                \\
            \infty
                &\text{if $t\in\bigl(r^\qp,\infty\bigr]$,}
        \end{cases}
\end{equation*}
and, finally, the left-continuous inverse $g_r^{-1}$ of $g_r$ obeys
\begin{equation*}
    g_r^{-1}(\tau) =
        \begin{cases}
            \left(\qp\tau\right)^{\frac{q}{p-q}}
                &\text{if $\tau\in\bigl[0,\pq r^{1-\qp}\bigr]$}
                \\[7pt]
            r^\qp
                &\text{if $\tau\in\bigl(\pq r^{1-\qp},\infty\bigr]$.}
        \end{cases}
\end{equation*}
Fix $r>0$ and $f\in\MRM$ supported on a set of measure not exceeding $r$ and satisfying $\nrm{f}_A\le1$. Note that then, in particular, $\int_\RR A(|f|)<\infty$.
We then have $f_*(t)\le r$ for every $t\ge 0$, and, consequently, $G^{-1}(f_*)=G_r^{-1}(f_*)$.
Hence,
\begin{align*}
    \nrm{f}_{p,q}^{q} & = q\int_{0}^{\infty} f_*(t)^\qp t^{q-1}\dd t = q\int_{0}^{\infty} G_r^{-1}(f_*(t)) t^{q-1}\dd t.
\end{align*}
By inequality~\eqref{E:OL-inequality}, applied to $G=G_r$ and $v(t)=t^{q-1}$, we thus get
\begin{equation}\label{E:embedding-orlicz-to-lorentz-4}
    \nrm{f}_{p,q}^{q} \le q\int_{0}^{\infty} g_r^{-1}\left(\frac{t^{q-1}}{\lambda a(t)}\right)t^{q-1}\dd t + \lambda q \int A(|f|)\quad\text{for every $\lambda>0$.}
\end{equation}
Fix $\varepsilon\in(0,\infty)$. We take $\lambda>0$ such that $\lambda q  < \varepsilon$.
Then, by~\eqref{E:Orlicz-balls}, we also have
\begin{equation}\label{E:embedding-orlicz-to-lorentz-5}
    \lambda q \int_\RR A(|f|) < \varepsilon.
\end{equation}
This $\lambda$ will be considered fixed from now on.
We shall now concentrate on the first summand of the right-hand side of inequality~\eqref{E:embedding-orlicz-to-lorentz-4}.
For $r>0$, we split the set $(0,\infty$) into two sets
\begin{equation*}
    S_r = \set*{t\in(0,\infty) : \frac{t^{q-1}}{\lambda a(t)} \le \pq r^{1-\qp}} \quad \text{and} \quad B_r=(0,\infty)\setminus S_r
\end{equation*}
and we estimate the integrals over these sets independently.
Due to the definition of $g_r^{-1}$, we have
\begin{align*}
    \int_{S_r}g_r^{-1}\left(\frac{t^{q-1}}{\lambda a(t)}\right)t^{q-1}\dd t
        = \int_{S_r}\left(\qp\frac{t^{q-1}}{\lambda a(t)}\right)^{\frac{q}{p-q}}t^{q-1}\dd t
        = c_{p,q,\lambda} \int_{S_r}\left(\frac{t^{q-1}}{a(t)}\right)^{\frac{q}{p-q}}t^{q-1}\dd t,
\end{align*}
where $c_{p,q,\lambda}$ is the constant involving only the parameters $p$, $q$ and $\lambda$.
Next, $S_r$ is contained in $(\inf S_r, \infty)$ and $a(t)\ge A(t)t$, and therefore we have
\begin{align*}
    \int_{S_r}g_r^{-1}\left(\frac{t^{q-1}}{\lambda a(t)}\right)t^{q-1}\dd t
        \le c_{p,q,\lambda} \int_{\inf S_r}^{\infty} \left(\frac{t^{q}}{A(t)}\right)^{\frac{q}{p-q}}t^{q-1}\dd t
        = c_{p,q,\lambda} \int_{\inf S_r}^{\infty} \left(\frac{t^{p}}{A(t)}\right)^{\frac{q}{p-q}}\frac{\d t}{t}.
\end{align*}
The sets $S_r$ are increasing in $r$ with respect to the set inclusion, and one has
$\lim_{r\to 0_+} \inf S_r = \infty$.
Indeed, if this was not true, then one would necessarily have ${t^{q-1}}/{a(t)}=0$ for some $t>0$, but that would contradict the fact that $a$ is finite valued.
Therefore, owing to \eqref{E:the-condition-fin},
\begin{equation}\label{E:embedding-orlicz-to-lorentz-7}
    \lim_{r\to0_+} \int_{S_r}g_r^{-1}\left(\frac{t^{q-1}}{\lambda a(t)}\right)t^{q-1}\dd t
        \le c_{p,q,\lambda} \lim_{R\to\infty}\int_{R}^{\infty} \left(\frac{t^{p}}{A(t)}\right)^{\frac{q}{p-q}}\frac{\d t}{t}
        = 0.
\end{equation}
To estimate the integral over $B_r$ define
\begin{equation*}
    \eta(t)=\sup_{s\in[t,\infty)}\frac{s^{q-1}}{a(s)}
    \quad\text{for $t\in[0,\infty)$},
\end{equation*}
and $\eta(\infty)=\lim_{t\to\infty}\eta(t)$.
Since $a$ is continuous, $\eta$ is also continuous and nonincreasing.
If we denote $\tau=\tfrac{\lambda p}{q} r^{1-\qp}$, we have
\begin{equation*}
    B_r \subseteq \set*{t\in(0,\infty):\eta(t)\ge \tau}
        \subseteq \left(0, \eta^{-1}(\tau)\right],
\end{equation*}
where $\eta^{-1}$ is the left-continuous inverse of $\eta$. Therefore
\begin{align*}
\int_{B_r}g_r^{-1}\left(\frac{t^{q-1}}{\lambda a(t)}\right)t^{q-1}\dd t
    & = r^\qp\int_{B_r}t^{q-1}\dd t
      \le r^\qp \int_{0}^{\eta^{-1}(\tau)} t^{q-1}\dd t
      = \frac{r^\qp}{q}\eta^{-1}\left(\frac{\lambda p}{q} r^{1-\qp}\right)^q.
\end{align*}
We claim that
\begin{equation}\label{E:embedding-orlicz-to-lorentz-8}
    \lim_{r\to 0_+} r^\qp\eta^{-1}\left(\frac{\lambda p}{q} r^{1-\qp}\right)^q = 0.
\end{equation}
This can be equivalently rewritten, using the change of variables $\tau=\tfrac{\lambda p}{q} r^{1-\qp}$ as
\begin{equation*}
    \lim_{\tau\to 0_+} \tau^{\frac{q}{p-q}}\eta^{-1}(\tau)^q = 0.
\end{equation*}
Let us now denote $\sigma=\eta^{-1}(\tau)$.
If $\sigma$ is bounded as $\tau\to 0_+$, we are done.
Therefore, in the sequel, we assume that $\sigma\to\infty$ as $\tau\to 0_+$.
By the continuity of $\eta$, we have that $\tau=\eta(\sigma)$.
Therefore, in order to prove~\eqref{E:embedding-orlicz-to-lorentz-8}, it suffices to show that
\begin{equation} \label{E:lim-with-sigma}
    \lim_{\sigma\to \infty} \eta(\sigma)\sigma^{p-q} = 0.
\end{equation}
Now, by the definition of $\eta$, we have
\begin{equation} \label{E:lim-with-sigma-bound}
\lim_{\sigma\to \infty} \eta(\sigma)\sigma^{p-q}
    =  \lim_{\sigma\to \infty} \sigma^{p-q} \sup_{t\ge \sigma}\frac{t^{q-1}}{a(t)}
    \le \lim_{\sigma\to \infty} \sup_{t\ge \sigma}\frac{t^{p-1}}{a(t)}.
\end{equation}
Since the integral in~\eqref{E:the-condition-fin} is convergent, it holds that $\lim_{t\to\infty}{t^{p}}/{A(t)}=0$ or equivalently, thanks to inequalities~\eqref{E:Young-basic-derivative-estimate}, $\lim_{t\to\infty}{t^{p-1}}/{a(t)}=0$.
Therefore, also
\begin{equation*}
    \lim_{\sigma\to \infty} \sup_{t\ge \sigma}\frac{t^{p-1}}{a(t)} = 0,
\end{equation*}
and~\eqref{E:embedding-orlicz-to-lorentz-8} follows by this and~\eqref{E:lim-with-sigma-bound}.
In conclusion, we showed
\begin{equation}\label{E:estimate-of-Br}
    \lim_{r\to0_+}\int_{B_r}g_r^{-1}\left(\frac{t^{q-1}}{\lambda a(t)}\right)t^{q-1}\dd t=0.
\end{equation}
Since $S_r\cup B_r=(0,\infty)$ for every $r>0$, we obtain, owing to the combination of~\eqref{E:embedding-orlicz-to-lorentz-7} and~\eqref{E:estimate-of-Br} that
\begin{equation*}
    \lim_{r\to 0_+} \int_{0}^{\infty} g_r^{-1}\left(\frac{t^{q-1}}{\lambda a(t)}\right)t^{q-1}\dd t = 0.
\end{equation*}
We now find $r$ small enough so that
\begin{equation}\label{E:embedding-orlicz-to-lorentz-9}
     q\int_{0}^{\infty} g_r^{-1}\left(\frac{t^{q-1}}{\lambda a(t)}\right)t^{q-1}\dd t < \varepsilon.
\end{equation}
Then it follows from~\eqref{E:embedding-orlicz-to-lorentz-4}, \eqref{E:embedding-orlicz-to-lorentz-5} and~\eqref{E:embedding-orlicz-to-lorentz-9} that
\begin{equation*}
    \|f\|_{p,q}^q<2\varepsilon.
\end{equation*}
Since $f$ was arbitrary, this implies~\eqref{E:embedding-orlicz-to-lorentz-3}. This establishes that $L^A\hrastar L^{p,q}$ and completes the proof.
\end{proof}

\begin{proof}[Proof of Theorem \ref{T:embeddings-orlicz-to-lorentz}]
We begin by proving part~\ref{en:convex}.
Since we have already shown sufficiency of the condition \eqref{E:the-condition-fin} in Theorem \ref{T:main}, and since \ref{en:OL-ace} trivially implies \ref{en:OL-emb}, it suffices to show the equivalences \ref{en:OL-emb} $\Leftrightarrow$ \ref{en:OL-mod} and \ref{en:OL-int} $\Leftrightarrow$ \ref{en:OL-mod}.

Assume first that $\mu(\RR)=\infty$. We first show the equivalence of~\ref{en:OL-emb} and~\ref{en:OL-mod}. Suppose that~\ref{en:OL-mod} holds and let $g$ be a non-zero function in $L^A$. Set $f={g}/{\nrm{g}_A}$. By the definition of the Luxemburg norm, one has $\int_{\RR}A(|f|)\dd\mu\le 1$, whence~\ref{en:OL-mod} implies that $\nrm{f}_{p,q}\le C$, or
\begin{equation*}
    \nrm{g}_{p,q} \le C\nrm{g}_A,
\end{equation*}
which yields~\ref{en:OL-emb}.

Now assume that~\ref{en:OL-emb} holds with the constant of the embedding denoted by $C$ and let $g$ be a non-zero function in $L^A$.
If $\varrho_A(g)=\infty$, then~\ref{en:OL-mod} holds trivially and there is nothing to prove.
Suppose, thus, that $\varrho_A(g)<\infty$.
Since the function ${g_*}/{\varrho_A(g)}$ is non-increasing on $(0,\infty)$ and $(\RR,\mu)$ is non-atomic, there exists a function $f\in\MM(\RR,\mu)$ such that $f_*={g_*}/{\varrho_A(g)}$.
Then, by~\eqref{E:Orlicz-distribution-formula},
\begin{equation*}
    \varrho_A(f) = \int_{0}^{\infty}a(t)f_*(t)\dd t = \frac{1}{\varrho_A(g)}\int_{0}^{\infty}a(t)g_*(t)\dd t = \frac{\varrho_A(g)}{\varrho_A(g)}=1.
\end{equation*}
Thus, by the definition of the Luxemburg norm, $\nrm{f}_{L^A}\le1$. Consequently, by~\ref{en:OL-emb}, one has $\nrm{f}_{p,q}\le C$. Therefore,
\begin{align*}
    \nrm{g}_{p,q} 
    = \left(q\int_{0}^{\infty}g_*(t)^{\frac{q}{p}}t^{q-1}\dd t\right)^{\frac1q}
    = \left(q\int_{0}^{\infty}\varrho_A(g)^{\frac{q}{p}}f_*(t)^{\frac{q}{p}}t^{q-1}\dd t\right)^{\frac1q}
    = \varrho_A(g)^{\frac1p}\nrm{f}_{p,q}\le C\varrho_A(g)^{\frac1p},
\end{align*}
and~\ref{en:OL-mod} follows.

We will now show the equivalence of~\ref{en:OL-int} and~\ref{en:OL-mod}. We first note that~\ref{en:OL-mod} can be rewritten as
\begin{equation*}
    \left(q\int_{0}^{\infty}f_*(t)^{\frac{q}{p}}t^{q-1}\dd t\right)^{\frac pq}
    \le C^p \int_{0}a(t)f_*(t)\dd t
\end{equation*}
for every $f\in\MM(\RR,\mu)$. This inequality is in turn equivalent to saying that
\begin{equation*}
    \left(q\int_{0}^{\infty}\varphi(t)^{\frac{q}{p}}t^{q-1}\dd t\right)^{\frac pq}
    \le C^p \int_{0}a(t)\varphi(t)\dd t
\end{equation*}
for every non-negative and non-increasing $\varphi$ on $(0,\infty)$. By~\cite[Remark~(i), page~148]{Saw:90}, this holds if and only if~\ref{en:OL-int} is satisfied. This completes the proof in the case $\mu(\RR)=\infty$. If $\mu(\RR)<\infty$, we modify the function $A$ near zero in order that
\begin{equation}\label{E:the-condition-prop-near-zero}
        \int_{0} \biggl( \frac{t^p}{A(t)} \biggr)^{\frac{q}{p-q}} \frac{\d t}{t} < \infty
\end{equation}
without changing the corresponding Orlicz space, and then follow the same argumentation as in the case when $\mu(\RR)=\infty$.

We continue by proving part~\ref{en:non-convex}.
We first recall that the equivalence of \ref{en2:OL-acLeb} and \ref{en2:OL-lim} is an immediate consequence of the classical result \cite[Chapter 5, Section 5.3, Proposition 7]{Rao:91}. Furthermore, \ref{en2:OL-acLeb} implies \ref{en2:OL-acLor} as, by our assumption on $q$, we have $L^p\hra L^{p,q}$. So, as $L^{p,q} \hra L^{p,\infty}$, it is enough to show that $L^A \hrastar L^{p,\infty}$ implies \ref{en2:OL-lim}.

Assume, thus, that $L^A \hrastar L^{p,\infty}$. Since $L^{\infty,\infty}=L^{\infty}$ and  $L^A\hrastar L^{\infty}$ is false for any Young function $A$, we must have $p<\infty$.
As $(\RR,\mu)$ is non-atomic, for each $t\in [0,\mu(\RR)]$, there is some $\mu$-measurable set $E_t$ such that $\mu(E_t)=t$. Define $f_t\in\MM$ by
\begin{equation*}
    f_t=\chi_{E_t} A^{-1}(\tfrac{1}{t}).
\end{equation*}
One has $\chi_{E_t}\to0$ $\mu$-a.e.~on~$\RR$ as $t\to0_+$. Thus, owing to $L^A\hrastar L^{p,\infty}$, we have
\begin{equation*}
    \lim_{t\to0_+} \sup_{\nrm{f}_{A}\le1}\nrm{f\chi_{E_t}}_{p,\infty}=0.
\end{equation*}
It is easily checked, see \eg~\cite[Eq.~9.23]{Kra:61}, that $\| f_t \|_{L^A}=1$ for each $t\in [0,\mu(\RR)]$. Hence, in particular,
\begin{equation*}
    \lim_{t\to0_+} \nrm{f_t\chi_{E_t}}_{p,\infty}=0.
\end{equation*}
But
\begin{equation*}
    \lim_{t\to 0_+} \| f_t\chi_{E_t} \|_{{p,\infty}}
    =\lim_{t\to 0} \sup_{s\in (0,\infty)} s^{\frac{1}{p}} A^{-1}(\tfrac{1}{t}) \chi_{(0,t)}(s)
    =\lim_{t\to 0_+} t^{\frac{1}{p}} A^{-1}(\tfrac{1}{t}),
\end{equation*}
so combining the estimates we get
\begin{equation*}
    \lim_{t\to 0_+} t^{\frac{1}{p}} A^{-1}(\tfrac{1}{t})=0.
\end{equation*}
Using the definition of the left-continuous inverse and properties of Young functions, this is readily seen to be equivalent to \ref{en2:OL-lim}.
\end{proof}

\begin{proof}[Proof of Theorem~\ref{T:embeddings-lorentz-to-orlicz}]
We start with Part~\ref{en:convex-dual}.
First, we show the equivalence of \ref{en:OL-emb-dual}, \ref{en:OL-int-dual} and \ref{en:OL-mod-dual} in the case $\mu(\RR)=\infty$.
We begin with the equivalence of \ref{en:OL-emb-dual} and \ref{en:OL-mod-dual}.
Assume that~\ref{en:OL-mod-dual} holds and let $f\in L^{r,s}$.
Assume that $f$ is a non-zero function and set $g={f}/{(C\nrm{f}_{r,s})}$.
Then~\ref{en:OL-mod-dual} implies that
\begin{equation*}
    \int_{\RR}B(|g|)\dd\mu \le 1,
\end{equation*}
whence $\nrm{g}_B\le1$, by the definition of the Luxemburg norm.
In other words, $\nrm{f}_B\le C\nrm{f}_{r,s}$, and statement~\ref{en:OL-emb-dual} follows.

Now assume that~\ref{en:OL-emb-dual} holds and denote by $C$ the constant of the embedding. Let $g\in L^{r,s}$ be a~non-zero function.
Since ${g_*}/{(C\nrm{g}_{r,s})^r}$ is a non-increasing function on $(0,\infty)$, there exists $f\in\MM(\RR,\mu)$ such that $f_*={g_*}/{(C\nrm{g}_{r,s})^r}$.
Then
\begin{equation*}
    \nrm{f}_{r,s}^s = s\int_{0}^{\infty}f_*(t)^{\frac{s}{r}}t^{s-1}\dd t
    = s \int_{0}^{\infty} \frac{g_*(t)^{\frac sr}}{C^s\nrm{g}_{r,s}^s}t^{s-1}\dd t = \frac{1}{C^s},
\end{equation*}
that is $\nrm{f}_{r,s}={1}/{C}$.
Thus, by~\ref{en:OL-emb-dual}, we have $\nrm{f}_B\le1$.
By the definition of the Luxemburg norm, one has $\varrho_B(f)\le 1$, or equivalently $\int_{0}^{\infty}b(t)f_*(t)\dd t\le 1$.
Consequently,
\begin{equation*}
    \int_{0}^{\infty}b(t)\frac{g_*(t)}{(C\nrm{g}_{r,s})^r}\dd t\le 1,
\end{equation*}
that is,
\begin{equation*}
    \int_{0}^{\infty}b(t)g_*(t)\dd t\le (C\nrm{g}_{r,s})^r,
\end{equation*}
which reads as
\begin{equation*}
    \left(\int_{\RR}B(g)\dd \mu\right)^{\frac1r}\le C\nrm{g}_{r,s},
\end{equation*}
and~\ref{en:OL-mod-dual} follows.

We continue by showing the equivalence of~\ref{en:OL-emb-dual} and~\ref{en:OL-int-dual}. Given $r,s$ such that $1<r<s\le\infty$, we set $p=r'$ and $q=s'$. By duality,~\ref{en:OL-emb-dual} is equivalent to $L^{\widetilde B}\hra L^{p,q}$. By Theorem \ref{T:embeddings-orlicz-to-lorentz} \ref{en:convex} (note that $1\le q<p<\infty$), this holds if and only if
\begin{equation}\label{E:embedding-orlicz-to-lorentz-nonconvex-dual-1}
    \int_{0}^\infty \biggl( \frac{t^p}{\widetilde B(t)} \biggr)^{\frac{q}{p-q}} \frac{\d t}{t} < \infty.
\end{equation}
Owing to~\eqref{E:Young-basic-derivative-estimate} and~\eqref{E:integral-representation-of-young-conjugate}, condition~\eqref{E:embedding-orlicz-to-lorentz-nonconvex-dual-1} reads as
\begin{equation}\label{E:embedding-orlicz-to-lorentz-nonconvex-dual-2}
    \int_{0}^\infty \biggl( \frac{t^p}{tb^{-1}(t)} \biggr)^{\frac{q}{p-q}} \frac{\d t}{t} < \infty.
\end{equation}
Since, by Fubini's theorem, we have
\begin{align*}
    \int_{0}^\infty \biggl( \frac{t^p}{tb^{-1}(t)} \biggr)^{\frac{q}{p-q}} \frac{\d t}{t}
    &\approx \int_{0}^\infty t^{\frac{(p-1)q}{p-q}-1}\int_{b^{-1}(t)}^{\infty} \tau^{-\frac{q}{p-q}-1}\d\tau\d t
        \\
    &= \int_{0}^\infty \tau^{-\frac{q}{p-q}-1} \int_0^{b(\tau)}t^{\frac{(p-1)q}{p-q}-1}\d t\d\tau
        \\
    &\approx \int_{0}^\infty \tau^{-\frac{q}{p-q}-1} b(\tau)^{\frac{(p-1)q}{p-q}}\d\tau
        \\
    &\approx \int_{0}^\infty \tau^{-\frac{q}{p-q}-1-\frac{(p-1)q}{p-q}} B(\tau)^{\frac{(p-1)q}{p-q}}\d\tau
        \\
    &=\int_{0}^{\infty}\left(\frac{B(t)}{t^r}\right)^{\frac{s}{s-r}}\frac{dt}{t}<\infty,
\end{align*}
in which `$\approx$' denotes equality up to a multiplicative constant.
Note that this computation works also for $s=\infty$, in which case we, once again, interpret the value ${s}/{(s-r)}$ as $1$.
Consequently, condition~\eqref{E:embedding-orlicz-to-lorentz-nonconvex-dual-1} coincides with~\eqref{E:the-condition-dual}, whence~\ref{en:OL-emb-dual} is equivalent to~\ref{en:OL-int-dual}, as desired.

Thus, we see that if $\mu(\RR)=\infty$ the statements \ref{en:OL-emb-dual}, \ref{en:OL-int-dual} and \ref{en:OL-mod-dual} are equivalent.
By appropriately modifying the function $B$ near zero without changing its behaviour near infinity, we obtain the equivalence also in the case $\mu(\RR)<\infty$.
This is analogous to the proof of Theorem~\ref{T:embeddings-orlicz-to-lorentz}.

Since \ref{en:OL-ace-dual} implies~\ref{en:OL-emb-dual}, it remains to show that if $\mu(\RR)<\infty$, then also~\ref{en:OL-emb-dual} implies~\ref{en:OL-ace-dual}. 
A~simple duality argument shows that~\ref{en:OL-emb-dual} is equivalent to $L^{\widetilde B}\hra L^{r',s'}$, and~\ref{en:OL-ace-dual} is equivalent to $L^{\widetilde B}\hrastar L^{r',s'}$.
Hence the claim follows from the equivalence of the statements~\ref{en:OL-emb} and~\ref{en:OL-ace} of Theorem~\ref{T:embeddings-orlicz-to-lorentz}~\ref{en:convex}.

Finally, we show Part~\ref{en-non-convex-dual}.
This assertion follows from Theorem~\ref{T:embeddings-orlicz-to-lorentz}~\ref{en:non-convex} and the fact that, for $p\in(1,\infty)$, $\lim_{t\to \infty} {\widetilde B(t)}/{t^{p'}}=\infty$ is equivalent to $\lim_{t\to \infty} {t^p}/{B(t)}=\infty$.
We show one of the implications as the other is analogous. Suppose that
\begin{equation}\label{E:lim-for-duality}
    \lim_{t\to \infty} {t^p}/{B(t)}=\infty    
\end{equation}
and assume, for contradiction, that $\lim_{t\to \infty} {\widetilde B(t)}/{t^{p'}}=\infty$ does not hold.
Using relation~\eqref{E:Young-basic-derivative-estimate} for $\widetilde B$, we obtain that there is some $K>0$ such that $b^{-1}(t)<Kt^{p'-1}$ for ever $t>0$.
From the definition of the left-continuous inverse, for every $t>0$, there exists some $\tau\ge 0$ such that
\begin{equation*}
    b(\tau)\geq t
    \quad \text{and}\quad
    \tau\leq K t^{p'-1}.
\end{equation*}
Therefore, since $b$ is non-decreasing, we conclude that $b(Kt^{p'-1})\geq t$ for all $t>0$,
contradicting~\eqref{E:lim-for-duality} once we use inequalities~\eqref{E:Young-basic-derivative-estimate} once again.
\end{proof}

\section*{Funding}

This work was supported by:
\begin{itemize}
\item \href{http://dx.doi.org/10.13039/501100001824}{Czech Science Foundation}, grant no.~\href{https://starfos.tacr.cz/en/projekty/GA23-04720S?query=sevyaadgulyq}{23-04720S};
\item \href{http://dx.doi.org/10.13039/100007397}{Charles University}, project no.~327321.
\end{itemize}

\end{document}